\documentclass{amsart}
\usepackage{amsfonts}
\usepackage{amssymb}
\usepackage{amsmath}

\newtheorem{theorem}{Theorem}[section]
\theoremstyle{plain}

\newtheorem{corollary}[theorem]{Corollary}

\newtheorem{definition}{Definition}[section]
\newtheorem{example}{Example}[section]

\newtheorem{proposition}{Proposition}[section]

\numberwithin{equation}{section}
\input{tcilatex}

\begin{document}
\title[Convex G-metric spaces]{Convexity in G-metric spaces and
approximation of fixed points by Mann iterative process}
\author{Isa Yildirim}
\address{Department of Mathematics, Faculty of Science, Ataturk University,
Erzurum, 25240, Turkey.}
\email{isayildirim@atauni.edu.tr}
\author{Safeer Hussain Khan}
\address{Department of Mathematics,\ Statistics and Physics, Qatar
University, Doha 2713, Qatar.}
\email{safeer@qu.edu.qa}
\subjclass[2000]{Primary 47H10 ; Secondary 65J15}
\keywords{Convex structure, Convex G-metric space, Mann iterative process,
convergence}

\begin{abstract}
In this paper, we first define the concept of convexity in $G$-metric
spaces. We then use Mann iterative process in this newly defined convex $G$%
-metric space to prove some convergence results for some classes of
mappings. In this way, we can extend several existence results to those
approximating fixed points. Our results are just new in the setting.
\end{abstract}

\maketitle

\section{\textbf{Introduction and preliminaries}}

The study of metric fixed point theory has been researched extensively in
the past two decades or so because fixed point theory plays a key role in
mathematics and applied sciences. For example, in the areas such as
optimization, mathematical models, and economic theories.

In 2005, Mustafa and Sims introduced a new class of generalized metric
spaces called $G$-metric spaces (see \cite{z12}, \cite{z13}) as a
generalization of metric spaces $(X,d).$ This was done to introduce and
develop a new fixed point theory for a variety of mappings in this new
setting. This helped to extend some known metric space results to this more
general setting. The $G$-metric space is defined as follows:

\begin{definition}
\cite{z13} Let $X$ be a nonempty set and let $G:X\times X\times X\rightarrow 
\mathbb{R}
^{+}$ be a function satisfying the following properties:

(i) $G(x,y,z)=0$ if $x=y=z$

(ii) $0<G(x,x,y)$ for all $x,y\in X,$ with $x\neq y$

(iii) $G(x,x,y)\leq G(x,y,z)$ for all $x,y,z\in X,$ with $z\neq y$

(iv) $G(x,y,z)=G(x,z,y)=G(y,z,x)=...,$ (symmetry in all three variables); and

(v) $G(x,y,z)\leq G(x,a,a)+G(a,y,z)$ for all $x,y,z,a\in X$ (rectangle
inequality )$.$

Then the function $G$ is called a generalized metric or more specifically, a 
$G$-metric on $X$, and the pair $(X,G)$ is called a $G$-metric space.
\end{definition}

\begin{example}
\cite{z12} Let $X=%
\mathbb{R}
\backslash \{0\}.$ Define $G:X\times X\times X\rightarrow 
\mathbb{R}
^{+}$ by%
\begin{equation*}
G(x,y,z)=\left\{ 
\begin{array}{c}
|x-y|+|y-z|+|x-z|;\text{ if }x,y,z\text{ all have the same sign} \\ 
1+|x-y|+|y-z|+|x-z|;\text{ otherwise}%
\end{array}%
\right.
\end{equation*}%
Then $(X,G)$ is a $G$-metric space.
\end{example}

\begin{proposition}
\cite{z13} Let $(X,G)$ be a $G$-metric space. Then for any $x,y,z,$ and $%
a\in X,$ it follows that

(i) if $G(x,y,z)=0,$ then $x=y=z,$

(ii) $G(x,y,z)\leq G\left( x,x,y\right) +G\left( x,x,z\right) $ $,$

(iii) $G\left( x,y,y\right) \leq 2G\left( y,x,x\right) ,$

(iv) $G\left( x,y,z\right) \leq G\left( x,a,z\right) +G\left( a,y,z\right) ,$

(v) $G\left( x,y,z\right) \leq 2/3\left( G\left( x,y,a\right) +G\left(
x,a,z\right) +G\left( a,y,z\right) \right) ,$

(vi) $G\left( x,y,z\right) \leq G\left( x,a,a\right) +G\left( y,a,a\right)
+G\left( z,a,a\right) .$
\end{proposition}

\begin{definition}
\cite{z13} Let $(X,G)$ be a $G$-metric space and $(x_{n})$ a sequence of
points of $X.$ A point $x\in X$ is said to be the limit of the sequence $%
(x_{n})$ if $lim_{n,m\rightarrow \infty }G(x,x_{n},x_{m})=0,$ and we say
that the sequence $(x_{n})$ is $G$-convergent to $x.$
\end{definition}

\begin{proposition}
\cite{z13} Let $(X,G)$ be a $G$-metric space. Then the following are
equivalent.

(i) $(x_{n})$ is $G$-convergent to $x.$

(ii) $G(x_{n},x_{n},x)\rightarrow 0$ as $n\rightarrow \infty .$

(iii) $G(x_{n},x,x)\rightarrow 0$ as $n\rightarrow \infty .$

(iv) $G(x_{m},x_{n},x)\rightarrow 0$ as $m,n\rightarrow \infty .$
\end{proposition}

In 2008, 2009 and 2010, Mustafa et al. (\cite{z3}, \cite{z4}, \cite{z1})
gave the following fixed point theorems on some classes of contractive
mappings defined on a $G$-metric space.

\begin{theorem}
\label{g}\cite{z3} Let $(X,G)$ be a complete $G$-metric space and let $%
T:X\rightarrow X$ be a mapping satisfying one of the following conditions:%
\begin{equation}
G\left( Tx,Ty,Tz\right) \leq aG\left( x,y,z\right) +bG\left( x,Tx,Tx\right)
+cG\left( y,Ty,Ty\right) +dG\left( z,Tz,Tz\right)  \label{1}
\end{equation}%
or%
\begin{equation*}
G\left( Tx,Ty,Tz\right) \leq aG\left( x,y,z\right) +bG\left( x,x,Tx\right)
+cG\left( y,y,Ty\right) +dG\left( z,z,Tz\right)
\end{equation*}%
for all $x,y,z\in X$ where $0\leq a+b+c+d<1.$ Then $T$ has a unique fixed
point $($say $u,$ i.e., $Tu=u),$ and $T$ is $G$-continuous at $u.$
\end{theorem}

\begin{theorem}
\label{g1}\cite{z4} Let $(X,G)$ be a $G$-metric space and let $%
T:X\rightarrow X$ be a mapping such that $T$ satisfies the following three
conditions:

$(1)$ $G\left( Tx,Ty,Tz\right) \leq aG\left( x,Tx,Tx\right) +bG\left(
y,Ty,Ty\right) +cG\left( z,Tz,Tz\right) $ for all $x,y,z\in X$ where $%
0<a+b+c<1,$

$(2)$ $T$ is $G$-continuous at a point $u\in X,$

$(3)$ there is $x\in X;$ $\{T^{n}x\}$ has a subsequence $\{T^{n_{i}}x\}$ $G$%
-converges to $u.$

Then $u$ is the unique fixed point of $T.$
\end{theorem}

\begin{theorem}
\label{g2}\cite{z1} Let $(X,G)$ be a complete $G$-metric space and let $%
T:X\rightarrow X$ be a mapping satisfying the condition%
\begin{equation}
G(Tx,Ty,Tz)\leq \alpha G(x,y,z)+\beta \left\{
G(y,Ty,Ty)+G(z,Tz,Tz)+G(x,Tx,Tx)\right\}  \label{2}
\end{equation}%
for all $x,y,z\in X,$ where $0\leq \alpha +3\beta <1.$ Then $T$ has unique
fixed point (say $u$), and $T$ is $G$-continuous at $u$.
\end{theorem}

\begin{theorem}
\label{g3}\cite{z1} Let $(X,G)$ be complete $G$-metric space and let $%
T:X\rightarrow X$ be a mapping satisfying the condition%
\begin{equation}
G(Tx,Ty,Tz)\leq \alpha G(x,y,z)+\beta \max \left\{
G(x,Tx,Tx),G(y,Ty,Ty),G(z,Tz,Tz)\right\}  \label{3}
\end{equation}%
for all $x,y,z\in X,$ where $0\leq \alpha +\beta <1.$ Then $T$ has unique
fixed point $($say $u),$ and $T$ is $G$-continuous at $u.$
\end{theorem}

Note that all above results deal with existence of fixed points without
finding or approximating them. The reason behind is the unavailablity of
convex structure in $G$-metric spaces by that time. Note that it is
imperative to have a convex structure in order to approximate fixed points
using various iterative processes, for example, Mann iterative process.

Keeping the above in mind, in this paper, we first define the concept of
convexity in $G$-metric spaces. We then use Mann iterative process in this
newly defined convex $G$-metric space to prove some convergence results for
approximating fixed points of some classes of mappings. In this way, several
existence results (including above Theorems 1.1-1.4) can be extended to
those approximating fixed points. Our results are just new in the setting.

In 1970, Takahashi \cite{Ta} introduced the notion of convex metric spaces
and studied the approximation of fixed points for nonexpansive mappings in
this setting. Later on, many authors (\cite{AR, Kim, Liu, yil}) discussed
the existence of fixed points and convergence of different iterative
processes for various mappings in convex metric spaces. .

We recall some definitions as follows:

\begin{definition}
\label{Cvx2}\cite{Ta} A convex structure in a metric space $(X,d)$ is a
mapping $W:X^{2}\times \lbrack 0,1]\rightarrow X$ satisfying, for all $%
x,y,u\in X$ and all $\alpha \in \lbrack 0,1],$%
\begin{equation}
d\left( W\left( x,y;\alpha \right) ,u\right) \leq \alpha d\left( x,u\right)
+\left( 1-\alpha \right) d\left( y,u\right) .  \label{y1}
\end{equation}%
The triplet $(X,d,W)$ is called a convex metric space.
\end{definition}

Modi et al. \cite{a} introduced convex structure in $G$-metric spaces as
follows.

\begin{definition}
\label{g7}\cite{a} Let $(X,G)$ be a $G$-metric space. A mapping $%
W:X^{3}\times (0,1]\rightarrow X$ is said to be a convex structure on $(X,G)$
if for each $(x,y,z,\lambda )\in X^{3}\times (0,1]$ and for all $u,$ $v\in X$
the condition%
\begin{equation*}
G(u,v,W(x,y,z,\lambda ))\leq \frac{\lambda }{3}G(u,v,x)+\frac{\lambda }{3}%
G(u,v,y)+\frac{\lambda }{3}G(u,v,z)
\end{equation*}%
holds. If $W$ is convex structure on a $G$-metric space $(X,G),$ then the
triplet $(X,G,W)$ is called a convex $G$-metric space.
\end{definition}

Using this definition, they gave some fixed point results for weakly
compatible mappings.

Continuing, we define convex structure in $G$-metric spaces in a differnt
way as follows. Our defintion is more natural than that due to Modi et al. 
\cite{a}. We do not divide $\lambda $ into three equal parts but let $%
\lambda $ and $\beta $ take any values in $[0,1]$ as far as their sum is 1.
We keep part of the domain as $X^{2}$ (and hence only two terms on the right
hand side) which is a better analog to the well-celebrated convexity of
Takahashi \cite{Ta} and simpler in calculations than taking $X^{3}.$ Here is
our definition of convex structure in a $G$-metric space.

\begin{definition}
Let $(X,G)$ be a $G$-metric space. A mapping $W:X^{2}\times I^{2}\rightarrow
X$ is termed as a convex structure on $X$ if $G(W(x,y;\lambda ,\beta
),u,v)\leq \lambda G(x,u,v)+\beta G(y,u,v)$ for real numbers $\lambda $ and $%
\beta $ in $I=[0,1]$ satisfying $\lambda +\beta =1$ and $x,y,u$ and $v\in X.$

A $G$-metric space $(X,G)$ with a convex structure $W$ is called a convex $G$%
-metric space and denoted as $(X,G,W).$

A nonempty subset $C$ of a convex $G$-metric space $(X,G,W)$ is said to be
convex if $W(x,y;a,b)\in C$ for all $x,y\in C$ and $a,b\in I.$\ 
\end{definition}

Next, we transform the Mann iterative process to a convex $G$-metric space
as follows.

\begin{definition}
Let $(X,G,W)$ be convex $G$-metric space with convex structure $W$ and $%
T:X\rightarrow X$ be a mapping. Let $\left\{ {\alpha }_{n}\right\} $ be a
sequence in $[0,1]$ for $n\in 
\mathbb{N}
.$ Then for any given $x_{0}\in X,$ the iterative process defined by the
sequence $\left\{ x_{n}\right\} $ as%
\begin{equation}
x_{n+1}=W\left( x_{n},Tx_{n};1-{\alpha }_{n},{\alpha }_{n}\right) ,\ \text{\
\ \ }n\in 
\mathbb{N}
,  \label{m}
\end{equation}%
is called Mann iterative process in the convex metric space $(X,G,W).$
\end{definition}

It follows from the structure of convex $G$-metric space that%
\begin{eqnarray*}
G(x_{n+1},u,v) &=&G(W\left( x_{n},Tx_{n};1-{\alpha }_{n},{\alpha }%
_{n}\right) ,u,v) \\
&\leq &\left( 1-{\alpha }_{n}\right) G(x_{n},u,v)+a_{n}G(Tx_{n},u,v).
\end{eqnarray*}

Now, having given the most needed definition of convex structure on a $G$%
-metric space and re-written the Mann iterative process (\ref{m}) in this
setting, we are able to transform the above mentioned existence results
(Theorems \ref{g}, \ref{g1}, \ref{g2} and \ref{g3}) to those approximating
fixed points through strong convergence. And this is what we are going to do
in the next section.

\section{\textbf{Main Results}}

The following is our first result which approximates the fixed points of the
mappings $(\ref{1})$ but naturally in a convex $G$-metric space.

\begin{theorem}
\label{main1} Let $(X,G,W)$ be a convex $G$-metric space with a convex
structure $W$ and let $T:X\rightarrow X$ be a mapping with a fixed point $u$
satisfying the condition $(\ref{1})$ for all $x,y,z\in X$ where $a,b,c,d$
are nonnegative real numbers such that $0\leq a+3b<1.$ Let $\left\{
x_{n}\right\} $ be defined iteratively by $(\ref{m})$ and $x_{0}\in X,$ with 
$\left\{ {\alpha }_{n}\right\} \subset \lbrack 0,1]$ satisfying $%
\tsum\limits_{n=0}^{\infty }{\alpha }_{n}=\infty .$ Then $\left\{
x_{n}\right\} $ converges strongly to a fixed point of $T.$

\begin{proof}
Since $u$ is a fixed point of the mapping $T,$ we have 
\begin{eqnarray}
G(x_{n+1},u,u) &=&G(W\left( x_{n},Tx_{n};1-{\alpha }_{n},{\alpha }%
_{n}\right) ,u,u)  \label{4} \\
&\leq &\left( 1-{\alpha }_{n}\right) G(x_{n},u,u)+a_{n}G(Tx_{n},u,u)  \notag
\\
&=&\left( 1-{\alpha }_{n}\right) G(x_{n},u,u)+a_{n}G(Tx_{n},Tu,Tu).  \notag
\end{eqnarray}%
Using the inequality (\ref{1}) for $G(Tx_{n},Tu,Tu)$,%
\begin{eqnarray}
G(Tx_{n},Tu,Tu) &\leq &aG(x_{n},u,u)+bG\left( x_{n},Tx_{n},Tx_{n}\right)
\label{5} \\
&&+cG\left( u,Tu,Tu\right) +dG\left( u,Tu,Tu\right)  \notag \\
&=&aG(x_{n},u,u)+bG\left( x_{n},Tx_{n},Tx_{n}\right)  \notag \\
&\leq &aG(x_{n},u,u)+b\left[ G(x_{n},u,u)+G\left( u,u,Tx_{n}\right) \right] 
\notag \\
&\leq &aG(x_{n},u,u)+b\left[ G(x_{n},u,u)+2G(Tx_{n},u,u)\right] .  \notag
\end{eqnarray}%
Thus from (\ref{5}) we have 
\begin{equation}
G(Tx_{n},Tu,Tu)\leq \frac{a+b}{1-2b}G(x_{n},u,u).  \label{6}
\end{equation}%
From the inequalities (\ref{4}) and (\ref{6}), we obtain 
\begin{eqnarray}
G(x_{n+1},u,u) &\leq &\left( 1-{\alpha }_{n}\right) G(x_{n},u,u)+a_{n}\frac{%
a+b}{1-2b}G(x_{n},u,u)  \label{7} \\
&=&\left[ 1-{\alpha }_{n}\left( 1-\frac{a+b}{1-2b}\right) \right]
G(x_{n},u,u)  \notag \\
&=&\left[ 1-{\alpha }_{n}\left( 1-\delta \right) \right] G(x_{n},u,u)
\end{eqnarray}%
where 
\begin{equation*}
\delta =\frac{a+b}{1-2b}.
\end{equation*}%
Note that $0\leq \delta <1$ and $1-2b\neq 0.$ Indeed, 
\begin{equation*}
a+3b<1\Longrightarrow a+b<1-2b\Rightarrow \frac{a+b}{1-2b}<1.
\end{equation*}

Moreover, if $1-2b=0,$ then from above calculations, $a+b<1-2b$ means $a<-b, 
$ which is contradiction to $a\geq 0.$

Inductively we get%
\begin{equation}
G(x_{n+1},u,u)\leq \tprod\limits_{k=0}^{n}\left[ 1-{\alpha }_{k}\left(
1-\delta \right) \right] G(x_{0},u,u).  \label{9}
\end{equation}%
As $\delta <1,{\alpha }_{k}\in \lbrack 0,1]$ and $\tsum\limits_{n=0}^{\infty
}{\alpha }_{n}=\infty ,$ it results that 
\begin{equation*}
\lim_{n\rightarrow \infty }\tprod\limits_{k=0}^{n}\left[ 1-{\alpha }%
_{k}\left( 1-\delta \right) \right] =0,
\end{equation*}%
which by (\ref{4}) implies%
\begin{equation*}
\lim_{n\rightarrow \infty }G(x_{n},u,u)=0.
\end{equation*}%
Hence the sequence $\left\{ x_{n}\right\} $ defined iteratively by $(\ref{m})
$ converges strongly to the fixed point of $T$.
\end{proof}
\end{theorem}

The following corollary gives approximation result for the of mappings used
inTheorem 1.3.

\begin{corollary}
Let $(X,G,W)$ be a convex $G$-metric space with a convex structure $W$ and
let $T:X\rightarrow X$ be a mapping with a fixed point $u$ satisfying the
following inequality%
\begin{equation}
G\left( Tx,Ty,Tz\right) \leq aG\left( x,y,z\right) +b\left\{ G\left(
x,Tx,Tx\right) +G\left( y,Ty,Ty\right) +G\left( z,Tz,Tz\right) \right\}
\label{9,1}
\end{equation}%
for all $x,y,z\in X$ where $0<a+3b<1.$ Let $\left\{ x_{n}\right\} $ be
defined iteratively by $(\ref{m})$ and $x_{0}\in X,$ with $\left\{ {\alpha }%
_{n}\right\} \subset \lbrack 0,1]$ satisfying $\tsum\limits_{n=0}^{\infty }{%
\alpha }_{n}=\infty .$ Then $\left\{ x_{n}\right\} $ converges strongly to a
fixed point of $T.$
\end{corollary}

\begin{proof}
If we let $b=c=d$ in condition (\ref{1}), then it becomes condition (\ref%
{9,1}) and the proof follows from Theorem \ref{main1}.
\end{proof}

We can also get easily the following more general result if we replace sum
with max in (\ref{9,1}) at a proper place.

\begin{corollary}
Let $(X,G,W)$ be a convex $G$-metric space with a convex structure $W$ and
let $T:X\rightarrow X$ be a mapping with a fixed point $u$ satisfying the
following inequality%
\begin{equation}
G\left( Tx,Ty,Tz\right) \leq aG\left( x,y,z\right) +b\max \left\{ G\left(
x,Tx,Tx\right) ,G\left( y,Ty,Ty\right) ,G\left( z,Tz,Tz\right) \right\}
\label{9,12}
\end{equation}%
for all $x,y,z\in X$ where $0<a+3b<1.$ Let $\left\{ x_{n}\right\} $ be
defined iteratively by $(\ref{m})$ and $x_{0}\in X,$ with $\left\{ {\alpha }%
_{n}\right\} \subset \lbrack 0,1]$ satisfying $\tsum\limits_{n=0}^{\infty }{%
\alpha }_{n}=\infty .$Then $\left\{ x_{n}\right\} $ converges strongly to a
fixed point of $T.$
\end{corollary}

We also give the following result corresponding to contractive condition. $(%
\ref{2}).$

\begin{theorem}
\label{main2} Let $(X,G,W)$ be a convex $G$-metric space with a convex
structure $W$ and let $T:X\rightarrow X$ be a mapping with a fixed point $u$
satisfying the following inequality%
\begin{equation}
G\left( Tx,Ty,Tz\right) \leq aG\left( x,Tx,Tx\right) +bG\left(
y,Ty,Ty\right) +cG\left( z,Tz,Tz\right)   \label{9,2}
\end{equation}%
for all $x,y,z\in X$ where where $a,b,c\ $are nonnegative real numbers  $%
0\leq a+b+c<1$ and $a<\frac{1}{2}.$ Let $\left\{ x_{n}\right\} $ be defined
iteratively by $(\ref{m})$ and $x_{0}\in X,$ with $\left\{ {\alpha }%
_{n}\right\} \subset \lbrack 0,1]$ satisfying $\tsum\limits_{n=0}^{\infty }{%
\alpha }_{n}=\infty .$ Then $\left\{ x_{n}\right\} $ converges strongly to a
fixed point of $T$.

\begin{proof}
From the proof of Theorem \ref{main1}, we know that 
\begin{equation}
G(x_{n+1},u,u)\leq \left( 1-{\alpha }_{n}\right)
G(x_{n},u,u)+a_{n}G(Tx_{n},Tu,Tu).  \label{9,5}
\end{equation}%
Using the property of $T$ for $G(Tx_{n},Tu,Tu)$,%
\begin{eqnarray}
G(Tx_{n},Tu,Tu) &\leq &aG\left( x_{n},Tx_{n},Tx_{n}\right) +bG\left(
u,Tu,Tu\right) +cG\left( u,Tu,Tu\right)   \label{10} \\
&\leq &a\left[ G(x_{n},u,u)+G\left( u,u,Tx_{n}\right) \right]   \notag \\
&\leq &a\left[ G(x_{n},u,u)+2G(Tx_{n},u,u)\right] .  \notag
\end{eqnarray}%
Thus, from (\ref{10}) we get%
\begin{equation}
G(Tx_{n},Tu,Tu)\leq \frac{a}{1-2a}G(x_{n},u,u).  \label{11}
\end{equation}%
From the inequalities (\ref{9,5}) and (\ref{11}), we have 
\begin{eqnarray*}
G(x_{n+1},u,u) &\leq &\left( 1-{\alpha }_{n}\right) G(x_{n},u,u)+a_{n}\frac{a%
}{1-2a}G(x_{n},u,u) \\
&=&\left[ 1-{\alpha }_{n}\left( 1-\frac{a}{1-2a}\right) \right] G(x_{n},u,u).
\end{eqnarray*}%
If we denote 
\begin{equation*}
\delta =\frac{a}{1-2a}
\end{equation*}%
then we have $0\leq \delta <1$ and 
\begin{equation*}
G(x_{n+1},u,u)\leq \left[ 1-{\alpha }_{n}\left( 1-\delta \right) \right]
G(x_{n},u,u).
\end{equation*}%
Indeed 
\begin{equation*}
a<\frac{1}{2}\Longrightarrow \frac{a}{1-2a}<1\ \text{\ and }1-2a\neq 0.
\end{equation*}%
In a way similar to the proof of above the theorem, we obtain%
\begin{equation*}
\lim_{n\rightarrow \infty }G(x_{n},u,u)=0.
\end{equation*}%
Hence the proof.
\end{proof}
\end{theorem}

\begin{corollary}
Let $(X,G,W)$ be a convex $G$-metric space with a convex structure $W$ and
let $T:X\rightarrow X$ be a mapping with a fixed point $u$ satisfying the
following inequality%
\begin{equation}
G\left( Tx,Ty,Tz\right) \leq k\left\{ G\left( x,Tx,Tx\right) +G\left(
y,Ty,Ty\right) +G\left( z,Tz,Tz\right) \right\}  \label{12}
\end{equation}%
for all $x,y,z\in X$ where $0<k<\frac{1}{3}.$ Let $\left\{ x_{n}\right\} $
be defined iteratively by $(\ref{m})$ and $x_{0}\in X,$ with $\left\{ {%
\alpha }_{n}\right\} \subset \lbrack 0,1]$ satisfying $\sum\limits_{n=0}^{%
\infty }{\alpha }_{n}=\infty .$ Then $\left\{ x_{n}\right\} $ converges
strongly to a fixed point of $T$.
\end{corollary}

\begin{proof}
Take $a=b=c=k$ in condition (\ref{9,2}) of Theorem \ref{main2}.
\end{proof}

\end{document}